\documentclass[a4paper, 10pt, conference]{ieeeconf}      
\IEEEoverridecommandlockouts                              
\usepackage{graphics} 
\usepackage{epsfig} 
\usepackage{amsfonts} 
\usepackage{times} 
\usepackage{amsmath} 
\usepackage{amssymb}  
\usepackage{dsfont}
\usepackage{color}
\usepackage{mathtools}

\newtheorem{thm}{Theorem}
\newtheorem{defi}{Definition}
\newtheorem{prop}{Proposition}

\newtheorem{cor}{Corollary}

\newtheorem{example}{Example}

\newtheorem{conjecture}{Conjecture}

\newcommand{\sst}{{<}}
\newcommand{\slt}{{>}}

\newcommand{\re}{{\mathbb R}}

\newcommand{\nat}{{\mathbb N}}
\newcommand{\n}{{\mathbb N}}

\newcommand{\argmin}{{\mbox{argmin}}}
\newcommand{\spann}{{\mbox{span}}}

\providecommand{\com[2]}{\begin{tt}[#1: #2]\end{tt}}

\providecommand{\mh[1]}{{\color{black}#1}}
\providecommand{\rmj[1]}{{\color{black}#1}}
\providecommand{\rmjj[1]}{{\color{black}#1}}

\providecommand{\pol[1]}{\pi_ #1}
\providecommand{\tayl[1]}{\tau_#1}

\providecommand{\lieder[2]}{{\mathcal{L}^{#1}_{#2}}} 

\newcounter{nameOfYourChoice}
\begin{document}
\title{\LARGE \bf  Non-local Linearization of\\Nonlinear Differential Equations via {\emph{Polyflows}}}
\author{ Rapha\"el M. Jungers$^{1}$ 
\thanks{$^{1}$ICTEAM Institute, Universit\'e catholique de Louvain, Belgium. raphael.jungers@uclouvain.be.  R. J. is a F.R.S.-FNRS Research Associate and a Fulbright Fellow. He is also supported by the French Community of Belgium, the Walloon Region and the Innoviris Foundation.}
Paulo Tabuada$^{2}$ 
\thanks{Electrical and Computer Engineering Department at the University of California, Los Angeles. tabuada@ucla.edu. The work of Paulo Tabuada was partially supported by the NSF award 1645824.}}

\maketitle
\thispagestyle{empty}
\pagestyle{empty}

\begin{abstract}
Motivated by the mathematics literature on the algebraic properties of so-called ``polynomial vector flows'', we propose a technique for approximating nonlinear differential equations by linear differential equations. Although the idea of approximating nonlinear differential equations with linear ones is not new, we propose a new approximation scheme that captures both local as well as global properties. This is achieved via a hierarchy of approximations, where the $N$th degree of the hierarchy is a linear differential equation obtained by \emph{globally} approximating the $N$th Lie derivatives of the trajectories. 

We show how the proposed approximation scheme has good approximating capabilities both with theoretical results and empirical observations.  In particular, we show that our approximation has convergence range at least as large as a Taylor approximation while, at the same time, being able to account for asymptotic stability (a nonlocal behavior).  We also compare the proposed approach with recent and classical work in the literature.
\end{abstract}

\section{Introduction}

The constantly increasing scale and complexity of modern engineering systems has led to a renewed interest in the development of analysis techniques for nonlinear systems.   There is a great variety of techniques, all of which arguably come with their advantages and weaknesses.  Let us mention: \textbf{i)} exact methods, that leverage particular algebraic structures in the system equations such as feedback linearizability or flatness (e.g., \cite{levine2009analysis}); \textbf{ii)} optimization methods, such as Lyapunov or sum-of-squares, that can produce satisfactory answers, but rarely come with guarantees of efficiency because of nonconvexity of most problems/systems (e.g., \cite{henrion2005positive});  \textbf{iii)} other methods relying on the approximation of nonlinear systems by linear ones with the goal of applying powerful techniques from linear systems theory. In this last family, let us mention linearization (e.g., nearby an equilibrium point) or infinite-dimensional approaches like the Carleman linearization or the Koopman-operator approach \cite{kowalski1997nonlinear,mauroy2016global,kowalski1991nonlinear}. This list is far from exhaustive, but as the reader will see, the ideas developed in this paper bear similarities with these three approaches, while at the same time trying to cope with their disadvantages.

Our goal is to approximate the solutions of a nonlinear   differential equation of the form:
\begin{eqnarray}
\label{ode} \dot x&=& f(x), 
\end{eqnarray}
where $f$ is a given function from $\re^n$ to $\re^n.$ We denote by $\psi(x_0,t)$ the solution of \eqref{ode} with initial condition $x_0$ at time $t=0$. We assume that $f$ is $C^\infty$ (i.e., infinitely differentiable). 
Our goal is to find an approximation that leads to a closed-form expression for $\psi(x_0,t)$. However, rather than approximating the trajectory $\psi(x_0,t)$, we wish to \emph{approximate the differential equation}~\eqref{ode} with another one, which we can analyze with exact techniques.
More precisely, we will approximate \eqref{ode} with (a projection of) a linear differential equation:
\begin{eqnarray}
\label{ode-polyflow} \dot z&=& Az, 
\end{eqnarray}
where $z\in \re^{N}$ and $A\in \re^{N\times N}$ for some $N\in \n$. 
 The goal is to account for non-local properties of the trajectory, as for instance, stability, limit cycles, etc. \rmjj{As we will see in Section \ref{sec-num}, our approximated systems are able to reproduce such global behaviours, and we are thus hopeful that on a longer term, one could extend} this technique to approximate feedback-linearization of nonlinear systems, or even approximate optimal control, by applying classical linear optimal control techniques on the approximating linear system \eqref{ode-polyflow} \rmjj{rather than on the initial system}.  
 
 \subsection{Related work}
The literature on nonlinear systems is large (see \cite{khalil1996noninear} for a general account) and techniques to analyze nonlinear dynamical systems are countless.  Closer  to our setting are techniques that approximate nonlinear systems by linear, but infinite dimensional systems.  Among them, Carleman Linearization \cite{kowalski1997nonlinear,gaude2001solving}, or the Koopman \rmj{approach \cite{mauroy2016global,rowley2009spectral,mauroy2013isostables}}, have recently attracted attention while dating back to the 1930s. The main idea in these methods is to represent the dynamical system as a linear operator acting on an infinite dimensional Hilbert space (the space of functions). This functional view brings linearity but requires working on infinite dimensional spaces. In practice, a particular basis for the infinite-dimensional Hilbert space is chosen, and then the linear operator describing the dynamics is truncated for numerical analysis purposes. Even though these techniques bear some similarity with ours, we do not perform truncation but rather a more detailed approximation directly resulting in a \rmjj{finite dimensional} linear differential equation. For instance, in \cite{mauroy2016global}, the approximation relies on the choice of a basis of polynomials in order to approximate the nonlinear behavior, which is not the case in our method. However we believe that there are interesting connections to be made with these approaches, and we leave for further work a more detailed analysis of these connections.\\
Other approaches rely on linearization nearby a particular point in the state space \cite{hartman}, or more generally aiming at generating a linear system by identifying coefficients with the Taylor expansion of the vector field $f$ \cite{krener1984approximate}.  Compared with these works, our approach differs in that one of our goals is to approximate the system globally and not at a particular point in the state space.\\
Finally, several works have analyzed \rmjj{particular} nonlinear systems that are \emph{actually equivalent} to linear systems modulo some change of coordinates (e.g., \cite{su1982linear} and~\cite{FeedbackLin}). Further extensions of this line of work considered the embedding of nonlinear systems into higher-dimensional linear systems~\cite{levine1986nonlinear}. While the motivation is the same, our work aims at analyzing \emph{arbitrary} nonlinear systems. \rmjj{It can be seen as an approximate version, projecting in some suitable sense arbitrary systems} onto such a nice one enjoying exact linearizability properties.

 \section{Lie derivatives and Taylor approximations}

 The main tool we will use in order to generate a suitable linear system \eqref{ode-polyflow} is the concept of Lie derivative, which we now review.

\begin{defi}[Lie derivative]
Consider the differential equation~\eqref{ode} and let $g:\re^n\rightarrow \re$ be a $C^\infty$ function.  For $N \in \n,$ the \emph{N-th Lie derivative} of $g$ with respect to~\eqref{ode} is defined as follows:
\begin{eqnarray}
\label{lieder} \lieder{0}{f}g=g &  & N=0,\\ 
\nonumber\lieder{1}{f}g=\frac{\partial g}{\partial x}f &  & N=1,\\ 
\nonumber \lieder{N}{f}g= \frac{\partial \lieder{N-1}{f}g}{\partial x}f &  & \forall N\geq 1.
\end{eqnarray}

We denote the evaluation of $\lieder{N}{f}g$ at $x_0$ by $\lieder{N}{f}g(x_0)$. If $g$ is a vector valued function, $g:\re^n\rightarrow \re^m$, we still denote by $\lieder{N}{f}g(x)$ the vector of the $m$ Lie derivatives, i.e., the vector $v$ such that $v_i=\lieder{N}{f}g_i(x).$
\end{defi}

Let us also review two very classical ways of approximating the solution of a differential equation. In doing so we denote by $1_K$ the identity function on a set $K$.
\begin{defi}[Taylor approximation]
The Taylor approximation $\tayl{N}$ at $t=0$, and of order $N\in \nat$, of the solution $\psi(x_0,t)$ of \eqref{ode} is the polynomial:
\begin{equation} \label{eq-taylor}
\tayl{N}(x_0,t) = \sum_{i=0}^N  \lieder{i}{f}1_{\re^n}(x_0)\frac{ t^i}{i!}.
\end{equation}
\end{defi}
By increasing $N$, the Taylor approximation $\tau_N$ will converge to $\psi(x_0,t)$ as long as $t$ belongs to the \emph{radius of convergence} which is defined as the largest number $R\in \mathbb{R}_{+0}$ such that for every $t\in ]-R,R[$ the series:
$$\sum_{i=0}^\infty  \lieder{i}{f}1_{\re^n}(x_0)\frac{ t^i}{i!},$$
converges.

We recall the Hadamard formula for the radius of convergence.

\begin{thm}[Hadamard]
The radius of convergence of the Taylor approximation $\tayl{N}(x_0,t)$ is given by:
\begin{equation}\label{hadamard} R^{-1}=\limsup{\left(\frac{1}{i!}\left\vert \lieder{i}{f}1_{\re^n}(x_0)\right\vert \right)^{1/i}}.\end{equation}
In case the system has dimension $n$ larger than one, the notation $|\cdot|$ above denotes the 1-norm (sum of the absolute values of the entries of a vector).
\end{thm}

A simpler linearization technique consists in retaining the linear part of~\eqref{ode} and eliminating all the nonlinear terms.

\begin{defi}[Linearization]
Denoting the \emph{Jacobian} of $f$ at $x^*\in \re^n$ by:
$$T_{x^*}f=\left.\frac{\partial f}{\partial x}\right\vert_{x=x^*}, $$ we define the \emph{linearization at $x^*\in\re^n$} of \eqref{ode} as the linear differential equation:
\begin{eqnarray}
\label{eq:linearization} \dot{(x-x^*)}=T_{x^*}f(x-x^*). 
\end{eqnarray}
\end{defi}

\section{Introducing polyflows}

Our work was first motivated by the mathematics literature where a great effort has been devoted to understand the properties of differential equations~\eqref{ode} for which the solution $\psi(x_0,t)$ is a polynomial function of the initial condition $x_0,$ hence the name \emph{polyflow}, a portmanteau obtained by blending the terms \emph{polynomial} and \emph{flow}  (see, e.g., \cite{van1995locally,levine1986nonlinear,bass1985polynomial,van1994locally,coomes1991linearization}).%
  We now provide an alternative definition of polyflows that makes the connection to our approximation problem clearer. 
\begin{defi}\label{def:polyflows}
The solution $\psi(x_0,t)$ of the differential equation~\eqref{ode} is a \emph{polyflow} if there exist $\ell\in \nat$ and an injective smooth map $\xi:\ \re^n \to \re^{\ell}$ such that $\xi\circ \psi(x_0,t)$ is the solution of a linear differential equation on $\re^\ell$ with initial condition $\xi(x_0)$.
\end{defi}
The map $\xi$ embeds the nonlinear differential equation into a linear one on a typically higher-dimensional state space. If the linear differential equation is denoted as in~\eqref{ode-polyflow}, the previous definition requires the existence of a matrix $A\in \re^{\ell\times\ell}$ such that $T_x \xi\cdot f=A\xi$. When this type of equality holds, the vector fields $f(x)$ and $Az$ are said to be $\xi$-related (see Def. 4.2.2 in~\cite{MTAA}). The theorem below provides a more constructive description of this equality:
\begin{thm} \cite{van1994locally}
The solution $\psi(x_0,t)$ of the differential equation~\eqref{ode} is a polyflow if and only if there exists $N\in\nat$ such that for every $n'\geq N$ and $k\in \{1,\hdots,n\}$ there exist $\lambda_{i,j}\in \re$ satisfying:
\begin{eqnarray} 
\label{AlgEq}
\qquad \lieder{n'}{f}f_k=\sum_{i\leq N,\ j\leq n}\lambda_{i,j}\lieder{i}{f}f_j.&& 
\end{eqnarray}
\end{thm}
That is, polyflows are dynamical systems endowed with a nilpotency property: after a  finite number of Lie derivatives, further differentiation only produces functions \rmjj{that are} trapped in a finite dimensional vector space. This very same idea appeared in the control literature (e.g.,~\cite{levine1986nonlinear}) related to the problem of embedding \rmjj{nonlinear} systems into, possibly higher dimensional, linear systems.
\begin{cor}\label{cor:polyflow-linearform}
If the solution $\psi(x_0,t)$ of the differential equation~\eqref{ode} is a polyflow, then, there exist $N\in \n$, coefficients $\Lambda_0,\dots, \Lambda_{N-1}\in \re^{n \times n}$, and $N-1$ functions $y_i:\re\rmjj{_+}\to\re^n$, $i\in\{1,\hdots,N-1\}$, such that:
\begin{equation}\label{eq:polyflow-linearform}
\begin{bmatrix}
\dot{\psi}\\ \dot y_1\\ \vdots \\ \dot y_{N-1}
\end{bmatrix} = 
\begin{bmatrix}
0& I &\dots &0\\
0 &\dots &\dots &0\\
0 &\dots &\dots &I\\
\Lambda_0 &\Lambda_1 &\dots &\Lambda_{N-1}
\end{bmatrix}
\begin{bmatrix}
\psi\\ y_1\\ \vdots \\ y_{N-1}
\end{bmatrix},
\end{equation}
with initial conditions: 
\begin{equation} \label{eq-init-cond} \psi(x_0,0)=x_0, \mbox{ and } y_i(0)=\lieder{i}{f}1_{\re^n}(x_0), i\in\{1,\hdots,N-1\}.\end{equation}
\end{cor}

The integer $N$ will play an important role in the remainder of the paper and, for this reason, we call a polyflow satisfying the conditions of Corollary \ref{cor:polyflow-linearform} a \emph{$N$-polyflow.} 

Using the language of Definition~\ref{def:polyflows} we see that: 
$$\xi=\begin{bmatrix}1_{\re^n}\\ \lieder{1}{f}1_{\re^n} \\ \vdots \\ \lieder{N-1}{f}1_{\re^n}\end{bmatrix},\qquad T_x \xi\cdot f(x)=A\xi(x),$$
where $A$ is the matrix defining the right hand side of the linear differential equation~\eqref{eq:polyflow-linearform}.

Corollary \ref{cor:polyflow-linearform} may seem to imply that the polyflow property is very restrictive: the trajectory of a dynamical system which is a polyflow must be the projection of the trajectory of a linear system. Therefore, nonlinear behavior is only encoded in the initial conditions.  However, we will give below both theoretical and empirical evidence that an arbitrary nonlinear differential equation is ``close'' to a polyflow (provided that the degree $N$ is taken large enough).

\section{The technique:\\{polyflow} approximation of nonlinear systems}
The Taylor approximation \eqref{eq-taylor} proceeds by approximating the solution through the computation of its first derivatives. Rather than approximating the solution of the ODE, we directly approximate the differential equation itself, thereby hoping to obtain a global description of its solutions. 

Our main idea leverages the observation that, when the right hand side $f$ of \eqref{ode} is a polynomial of degree $d$, its solution is a $N$-polyflow if and only if it belongs to some lower dimensional manifold in the space of polynomials of degree $d$.  Indeed, one can see that in the characterization \eqref{AlgEq}, one can restrict the equalities to $n'=N$ (see Proposition \ref{prop-lie-derivatives} below), and these are algebraic conditions on the coefficients of the polynomial $f$.  Since polyflows have the nice property that they allow for a closed-form formula for their solutions, a natural idea is to ``project'' our system on the closest polyflow. We formalize this in the next two definitions:

\begin{defi}[Polyflow Approximation]\label{def:polyflow-approx}
Consider the differential equation~\eqref{ode} and a compact set $K\subset \re^n$. 
We define a \emph{N-th polyflow approximation of~\eqref{ode} on $K,$} to be any linear differential equation, as in \eqref{eq:polyflow-linearform}, where the matrices $\Lambda_i\in \re^{n\times n}$ are obtained by approximating $\lieder{N}{f}1_K$ with:
\begin{equation}
\label{eq:polyflow-approx}  \sum_{i=0}^{N-1} \Lambda_i \lieder{i}{f}1_K.
\end{equation}
\end{defi}

In the above definition, nothing is said on how to compute the parameters $\Lambda_i$ nor the initial conditions even though the quality of the approximation will depend on these choices. In, fact, it is not clear to the authors, at the moment, how to best define them.  Most probably, there is no unique ``right'' choice, but the wisest strategy will depend on what precisely is the final objective, as is often the case when one resorts to approximations. \rmj{We present here some natural strategies, even though others could be introduced.}

\begin{defi}[Numerical Strategies] \label{def:projections}
Consider the differential equation~\eqref{ode} and a compact set $K\subset \re^n$. 
The coefficients $\Lambda_i$ in Definition \ref{def:polyflow-approx} can be defined in the following ways:
\begin{enumerate}
\item \label{def:approx-1}  For any given norm $|\cdot |$ define: 
\begin{equation}\label{eq-approx-pf}\tilde f_N :=\argmin_{g\in S}\left\{\left\vert g-\lieder{N}{f}1_K\right\vert\right\}, 
\end{equation} where $S=\spann_{\re}{\{\lieder{i}{f}1_K:\ i\leq N-1\}}$ and $\Lambda_i$ is such that $\tilde f_N =\sum\Lambda_i \lieder{i}{f}1_K.$
\item \label{def:approx-1.bis} For any given norm $|\cdot |$ define: 
$$\tilde g := \argmin_{g\in S}\left\{\left\vert\lieder{N}{g}1_K-\lieder{N}{f}1_K\right\vert\right\}$$
where $S$ is the set of vector fields for which the solution of the corresponding differential equations are polyflows and $\Lambda_i$ is such that $\lieder{N}{\tilde g}{1_K} =\sum\Lambda_i \lieder{i}{\tilde g}{1_K}.$
\setcounter{nameOfYourChoice}{\value{enumi}}
\end{enumerate}

If one represents a polyflow approximation like in \eqref{eq:polyflow-linearform}, one has the following alternative definitions for the initial conditions: \begin{enumerate}
\setcounter{enumi}{\value{nameOfYourChoice}}
\item \label{def:approx-2}  For any $i\in\{1,\hdots,N-1\}$, $y_i(0)=\lieder{i}{f}1_K(x_0);$
\item \label{def:approx-2.bis} (In case Item \ref{def:approx-1.bis} was chosen above) Denoting $\tilde g$ as above,  $y_i(0)=\lieder{i}{\tilde g}{1_K}(x_0).$
\end{enumerate}
In Item \ref{def:approx-1.bis} above, we pre-specify the degree of the vector field $\tilde g,$ restrict it to be a polyflow, and then find such an optimum $\tilde g.$ In Item \ref{def:approx-1}, the approach is bolder: we only look at the Lie derivatives of the actual field $f,$ and directly project the $N$th Lie derivative onto the space generated by the previous ones. Unless specified otherwise, in our numerical experiments, we take Items \ref{def:approx-1} and \ref{def:approx-2}, together with the infinity norm: 
$$|f|:=\sup_{x\in K}{\{f(x)\}}.$$\end{defi}
\rmj{In our numerical experiments, we solve the approximation problem by discretizing the compact set $K,$ and solving the approximation problem on the obtained finite set of points. This can be done with standard Linear Programming. The approach scales well, as at step $N,$ it only requires to project $n$ polynomial scalar functions on a subspace spanned by $n(N-1)$ polynomial functions. The number of necessary discretization points in order to reach a specified accuracy grows exponentially with the dimension of the state-space, though. Also, note that the degrees of the polynomial functions obtained at step $N$ grow exponentially with $N.$}

We finish this section by providing a detailed analysis of two examples.  The first one is the trivial case where $f$ is linear.
\begin{example}
Consider the scalar linear system: 
$$ \dot x = - \lambda x.$$  
For any $N,$ there are infinitely many solutions to the approximation problem \eqref{eq-approx-pf} \rmjj{and the obtained approximation is exact} (because all the Lie Derivatives are linear functions).  However, there is one which is strictly better than the others: the one such that the polyflow has only one eigenvalue, equal to the eigenvalue of the local linearization, namely $-\lambda.$  With this choice, the coefficients of the polyflow in \eqref{eq:polyflow-linearform} are \footnote{The notation $\begin{pmatrix} N\\i \end{pmatrix}$ denotes the binomial coefficient.}: 
$$\Lambda_i = -|\lambda|^{N-i} \begin{pmatrix} N\\i \end{pmatrix},$$ 
and all the successive polyflow approximations lead to the exact trajectory.
\end{example}
\begin{example}\label{ex2}
We consider the two-dimensional system: 
\begin{equation} \dot x = \begin{bmatrix} x_1 +2 x_2^2 +x_2^3 -x_2^4 \\ -x_2 \end{bmatrix}. \end{equation} 
\rmj{The system is nonlinear, but as it turns out, it is a $4$-polyflow. It is thus equivalent to a linear system, and, as a consequence, our approximation becomes exact after finitely many steps.  We show this on Figure \ref{fig-ex-poly} (for  $x_0=(1,1)$).}

Let us detail the procedure for $N=1.$ First, the forms $x_1 +2 x_2^2 +x_2^3 -x_2^4 $ and $-x_2$ are projected\footnote{We project numerically by discretizing the state space on $[0,2]^2$ with discretization step $0.2.$ We minimize the $1$-norm of the error vector as previously stated. Since the projection is based on the whole compact set of interest ($[0,2]^2$ in this case) it captures global information about the system.} on $V=\spann_{\re}{\{x_1,x_2\}}.$ We retrieve the coefficients defining the projection on the linear space $V$ and obtain the approximating linear system:
\begin{equation} 
\dot x = \begin{bmatrix} x_1 +3.3143 x_2 \\ 
-x_2 \end{bmatrix}. 
\end{equation} 
Observe that, $\dot x_2$ being a linear function in the true system, the projection is exact for this variable. Finally we compute the initial conditions $\lieder{i}{f}1_K(x_0),\ i\sst N,$ and we obtain the approximate trajectory, represented in magenta in Figure \ref{fig-ex-poly}.  We do similarly for $N=2$ and $N=3.$

 Now, since the true system is a $4$-polyflow, the projection is exact for both variables at the fourth step (N=4), and the approximation is perfect, as one can see on Figure \ref{fig-ex-poly}. \end{example} 
\begin{figure}
\centering
\begin{tabular}{cc}
\includegraphics[scale = 0.15]{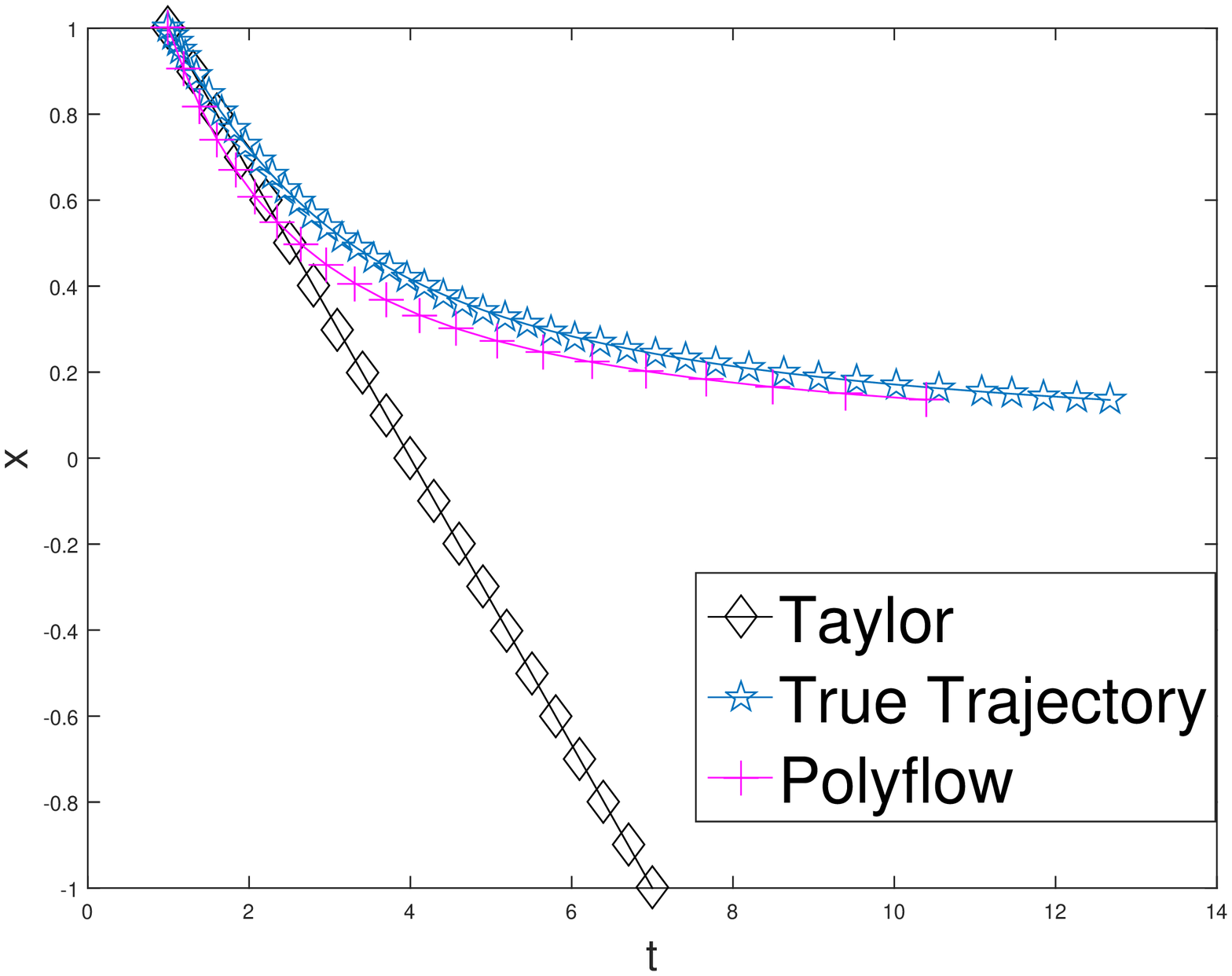}&
\includegraphics[scale = 0.15]{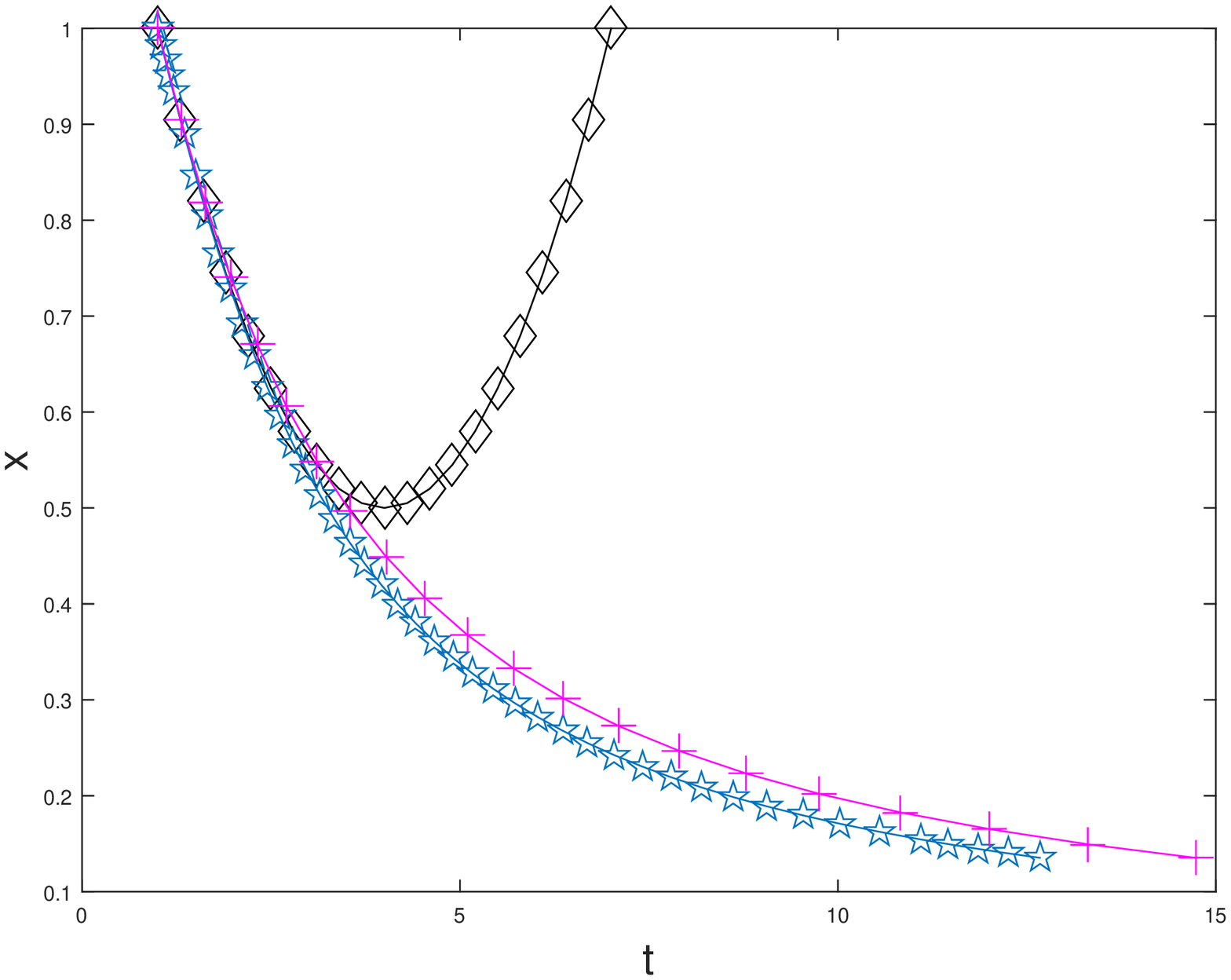} \\N=1&N=2\\ \\
\includegraphics[scale = 0.15]{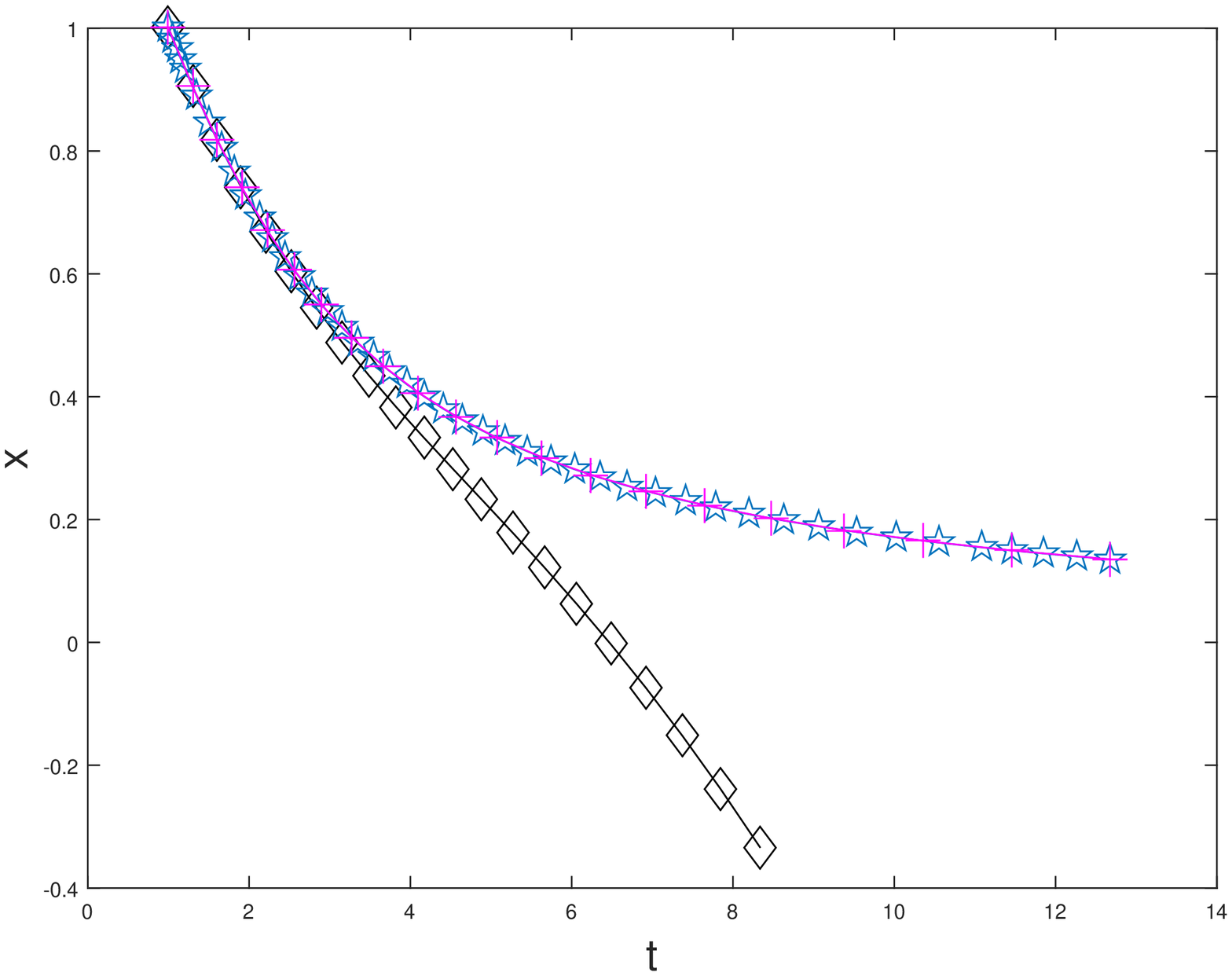}&
\includegraphics[scale = 0.15]{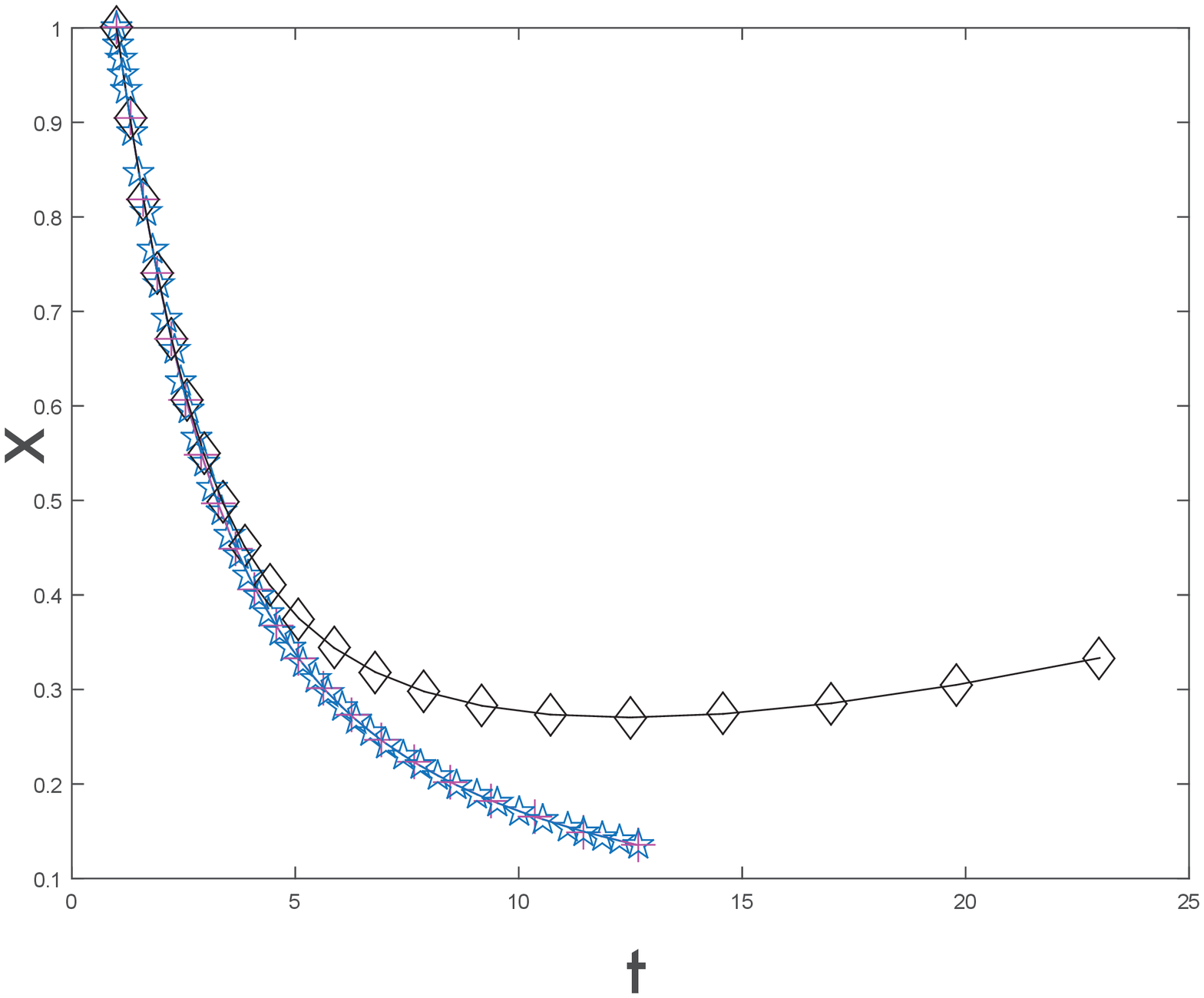}\\N=3 & N=4
\end{tabular}
 \caption{\rmjj{Example \ref{ex2}:} The approximation scheme (with $N=1,2,3,4$) applied on a two-dimensional $4$-polyflow (only the $x_1$-coordinate is represented). The fourth step provides a zero error, because the system is a $4$-polyflow. One can see that even for $N=1,$ our approximation retrieves stability, contrary to the Taylor approximation. As we observed in all our experiments, the approximation performance is already very good for small $N,$ and the error is barely noticeable for $N=3.$}
\label{fig-ex-poly}\end{figure}
\section{Theoretical results}

In this section we provide theoretical arguments showing that our procedure comes with good approximation properties.  We show that our approximation strategy can have performance guarantees that are at least as good as Taylor's approximation, while at the same time encapsulating asymptotic stability. \rmjj{Thus, our solution recovers an asymptotic property (i.e. asymptotic stability), but also satisfies the local properties of the Taylor approximation, \emph{even if the point where the Taylor approximation is done is different from the equilibrium.}} Our proof is an existence proof: we exhibit a theoretical construction for an approximating linear differential equation, which has the above-mentioned good properties.  We emphasize that the construction in the proof is taken as simple as possible and, in particular, does not fully exploit the knowledge of the righthand side $f$ in~\eqref{ode} as proposed in Definition \ref{def:projections}.  Thus, in practice, an optimized choice for the approximating linear differential equation can have an even better behavior than the one theoretically proven here, and indeed we do observe that our approximation outperforms the Taylor approximation in practice (see the examples above, and in Section \ref{sec-num}). At the end of this section we discuss possible extensions.
\subsection{Main result}
%



\begin{thm}\label{thm-stab}
Let $\psi(x_0,t)$ be the solution of~\eqref{ode} with initial condition $x_0$ at time $t=0$, let $\tayl{N}(x_0,t)$ be the Taylor approximation of order $N$ of $\psi(x_0,t)$, and let $R(x_0)$ be its radius of convergence. There exists a sequence of $N$-polyflow approximations $\pol{N}$ satisfying:
\begin{equation}\label{eq-mainthm-eq1}
\forall N\in\nat,\quad \lim_{t\rightarrow \infty} \pol{N}(x_0,t)=0, 
\end{equation} and such that for any initial condition $x_0\in\re^n$ and for any $t\in ]-R(x_0),R(x_0)[$ the successive approximations of the solution tend pointwise to the true solution, that is: 
\begin{equation}\label{eq-mainthm-eq2}
\forall t\in ]-R(x_0),R(x_0)[,\quad \lim_{N\rightarrow \infty} \pol{N}(x_0,t)=\psi(x_0,t). \end{equation}
\end{thm}

Although Theorem~\ref{thm-stab} does not require the existence of an asymptotically stable equilibrium, this is one of the scenarios where it will be most useful since the approximating polyflows $\pi_N$ are guaranteed to satisfy~\eqref{eq-mainthm-eq1}.

In the proof of the theorem, we express the approximating trajectory, solution of the polyflow approximation, as a Laurent series.  Then, since the first $N$ derivatives of this trajectory at $t=0$ are equal to those of the Taylor polynomial, one observes that convergence of the polyflow approximations to the true value in the radius of convergence, is equivalent to convergence of the remainder to zero. Then, one only needs to bound the larger derivatives at $t=0$ of the polyflow approximation (i.e., of order larger than $N$); \rmjj{this is done by expressing these derivatives as} the solution of a recurrence equation.	We start by recalling a few technical results.


\begin{prop}\label{prop-recurrence-bound}
Given a recurrence equation in $\re^m$:
\begin{equation}\label{eq-recurrence} 
X_{n}=\sum_{0\leq i\leq N-1} k_i X_{n-N+i} 
\end{equation} 
where  $k_i\in \re^{m \times m},$ such that every entry in the matrices $k_i$ has absolute value bounded by $K\in \re_+,$ one has\footnote{If $X_n \in \re^m$ are vectors (i.e., $m\slt 1$) we use $|X|:=\sum_{1\leq j\leq m}|X(j)|$.}: $$\forall n\geq N-1,\ |X_{n}| \leq  (mK+1)^{n-N+1}\sum_{0\leq i\leq N-1}{|X_i|}. $$
\end{prop}
\begin{proof}[Sketch]
We first introduce $s_n$ defined by:
$$s_n:=\sum_{0\leq i\leq n}|X_i|.$$ The righthand side in Equation \eqref{eq-recurrence} can be bounded thanks to the following inequalities: 
$$\left\vert\sum_{0\leq i\leq N-1} k_i X_{n+i}\right\vert\leq m\sum_{0\leq i\leq N-1} K |X_{n+i}| \leq mK s_{n+N-1}.$$ 
Substituting in \eqref{eq-recurrence} we obtain: 
$$|X_n| \leq mKs_{n-1}. $$  
Now, $s_n= s_{n-1} + |X_n|$ and thus $s_n \leq L s_{n-1}$ where $L:=mK+1$ for concision. Finally, we have $s_n \leq s_{N-1} L^{n-N+1}$ and the result follows. 
\end{proof}
Our second ingredient is an elementary property of linear differential equations:
\begin{prop} \label{prop-lie-derivatives}
Given a linear differential equation, as in \eqref{eq:polyflow-linearform}, for any $t\geq 0,$ and for any $i\geq N$ the $i$th derivative $\phi^{(i)}(x_0,t)$ of the solution $\phi(x_0,t)$ satisfies: 
$$\phi^{(i)}(x_0,t) = \sum_{0\leq j\leq N-1}{\Lambda_j \phi^{(i-N+j)}(x_0,t)}.$$
\end{prop}
\rmjj{The proof is immediate and we omit it.}
We are now in position to prove Theorem \ref{thm-stab}:

\begin{proof}[Proof of Theorem \ref{thm-stab}]\\
We first define our sequence of $N$-polyflow approximations (recall that $n$ is the dimension of the initial differential equation~\eqref{ode}): \\
At each step $N,$ we define the coefficients $\Lambda_{N,i}\in \re^{n \times n},$ $i\in\{ 0,N-1\}$ of the approximating linear differential equation so that the eigenvalues have negative real parts, and the $\Lambda_{N,i}$ are bounded by a uniform constant $K$ independent \rmjj{of} $N.$ \\
This is easily done, e.g., by choosing $\Lambda_{N,i}$ to be multiples of the identity $k_iI,$ and fixing all eigenvalues equal to a constant $-\epsilon,$ for $\epsilon$ small enough. One then computes $k_i$ by identifying them with the coefficients of the characteristic polynomial $(\lambda+\epsilon)^N.$\\ Since the eigenvalues are explicitly chosen with negative real parts, it is clear that the system is asymptotically stable. Thus, the first part of the theorem (equation \eqref{eq-mainthm-eq1}) is easily satisfied.

We now prove equation \eqref{eq-mainthm-eq2}, that is, the sequence of approximations converges in a region as large as the radius of convergence of the Taylor approximation.
 Let us fix  \mbox{$x_0 \in \re^n$} and a time $t\in[0,R(x_0)[$ (for simplicity we consider positive times only). First, we remark that the solution of the $N$-polyflow approximating system, satisfies:

\begin{eqnarray}  \pi_N(x_0,t)&=& \sum_{i=0}^\infty \pi_N^{(i)}(x_0,0) \frac{t^i}{i!}\\
\label{eqn-taylor-polyflow} &=& \sum_{i=0}^{N-1} \psi^{(i)}(x_0,0) \frac{t^i}{i!} +\sum_{i=N}^\infty \pi_N^{(i)}(x_0,0) \frac{t^i}{i!}.
\end{eqnarray}

The first term of the above sum tends to $\psi(x_0,t)$ when $N$ tends to infinity (indeed, this first term is precisely the $N$th Taylor approximation, \rmj{since we chose the initial conditions accordingly}).  Thus, we have to prove that the second term tends to zero.

By Proposition \ref{prop-lie-derivatives}, the coefficient $\pi_N^{(i)}(x_0,0) $ in the second term can be obtained via the recurrence relation:

$$\pi_N^{(k)}(x_0,0) = \sum_{i=0}^{N-1} \Lambda_{N,i} \pi_N^{(k-N+i)}(x_0,0).$$  
Thus, applying Proposition \ref{prop-recurrence-bound}, we have that (again, we introduce the constant $L:=nK+1$ for concision):
\begin{equation}\label{eq-bound-taylorcoeffs}
|\pi_N^{(k)}(x_0,0)|\leq K'_N (nK+1)^{k-N+1}:= K'_N L^{k-N+1},
\end{equation} 
where $K$ is a uniform upper bound on the absolute values of the entries of the $\Lambda_{N,i}$ and: 
$$K'_N:=\sum_{0\leq k\leq N-1} |\pi_N^{(k)}(x_0,0)|.$$ 
We now bound the value $K'_N$.

Since $t$ is in the radius of convergence of the Taylor approximations of $\psi(x_0,\cdot),$ by the Hadamard formula \eqref{hadamard}, we have that, for $\epsilon$ small enough, and $k\leq N-1,$ $$|\pi_N^{(k)}(x_0,0)|=|\psi^{(k)}(x_0,0)|\leq C k!/(t+\epsilon)^{k},$$ for some constant $C,$ independent of $N.$  Supposing $N$ large enough, the quantity $k!/(t+\epsilon)^k,\ k\leq N-1$ is bounded by its value for $k=N-1.$ Thus, the sum $K_N'$ can be bounded as follows:
\begin{eqnarray} K_N'&=&\sum_{0\leq k\leq N-1}|\pi_N^{(k)}(x_0,0)|\\ &\leq & \sum_{0\leq k\leq N-1} C k!/(t+\epsilon)^{k}\\ &\leq& NC (N-1)!/(t+\epsilon)^{N-1} .\end{eqnarray}

Plugging these in Equation \eqref{eq-bound-taylorcoeffs}, 
we obtain the following bound on the coefficients of the Laurent series \eqref{eqn-taylor-polyflow}  $$\forall k, |\pi_N^{(k)}(x_0,0)|\leq C N!/(t+\epsilon)^{(N-1)} L^{k-N+1}. $$

We are now able to finish the proof. The righthand term in Equation \eqref{eqn-taylor-polyflow} can be bounded as follows, for $i\geq N:$
\begin{eqnarray}
\left\vert\pi_N^{(i)}(x_0,0)\right\vert \frac{t^i}{i!}& \leq&  C  \frac{N! L^{i-N+1}t^i}{(t+\epsilon)^{(N-1)}i!}\\ 
&\leq & C \left(\frac{t}{t+\epsilon}\right)^{N-1} \cdot{} \\ 
& &\frac{N!}{i\dots (i-N+2)} \frac{(Lt)^{i-N+1}}{(i-N+1)!}\\ 
&\leq &C \left(\frac{t}{t+\epsilon}\right)^{N-1} \frac{(Lt)^{i-N+1}}{(i-N+1)!}.\end{eqnarray}
Finally:
\begin{equation}\label{eq-sum-thm}
\left\vert\sum_N^\infty \pi_N^{(i)}(x_0,0) \frac{t^i}{i!} \right\vert\leq  C \left(\frac{t}{t+\epsilon}\right)^{N-1} \exp{(Lt)} ,\end{equation} which tends to zero when $N$ tends to $\infty$ and the proof is done.
\end{proof}

\subsection{Discussion}
In fact, in the above theorem, at each step $N,$ we place the $N$ eigenvalues of the polyflow arbitrarily, without leveraging our knowledge of the system.  Obviously, there should be an optimal way of choosing the coefficients, so that the mid-term approximation (i.e., for $R(x_0) \leq t < \infty$) performance is optimized while ensuring that the asymptotic behavior is stable and the successive approximations converge on $[0,R[$. We conjecture that such a procedure would enable convergence, not only on the Taylor radius of convergence, but on the \emph{entire} interval of time $t\in [0,\infty]:$
\begin{conjecture} \label{conj-stab}
Consider the system \eqref{ode}, and suppose that $x=0$ is an asymptotically stable equilibrium.  There exists a way of computing the polyflow coefficients from $f$, on some neighborhood $K$ of $0$ such that, for all $x_0\in K,$ the polyflow approximations satisfy
$$ \forall t\geq 0,\  \lim_{N\rightarrow \infty} {\pol{N}(x_0,t)}=\psi(x_0,t).$$
\end{conjecture}
We suspect that these coefficients should be based on the Taylor coefficients of $f(x),$ in a similar way as in \cite{krener1984approximate}.  
As a matter of fact, it is well known from approximation theory that linear combinations of exponentials are dense in the sets of functions on a compact~\cite{NAT}, and then one can indeed approximate arbitrarily well any trajectory with a solution of a linear system. However, firstly, in our setting, we do not allow arbitrary coefficients for the linear combinations (these coefficients are determined by the initial conditions as in \eqref{eq-init-cond}) and secondly, we would like to compute the polyflow coefficients \emph{implicitly,} that is, from the sole knowledge of $f$ and not from the  knowledge of the trajectories themselves.

%

\section{Numerical examples}\label{sec-num}
In this section we report numerical examples on one- and two-dimensional systems. 
\subsection{One-dimensional systems}
We observe in practice, for all the stable polynomial systems that we tried, that indeed all the polyflow approximations are stable, provided that we approximate the vector field in a compact set contained in the interior of the basin of convergence. This tends to confirm Conjecture \ref{conj-stab}.  As an example, we represent in Figure \ref{fig-1D} a one dimensional nonlinear (cubic) system and its successive polyflow approximations.  One can see that the approximation works remarkably well, even for small values of $N.$
\begin{figure}
\begin{center}
\includegraphics[scale = 0.3]{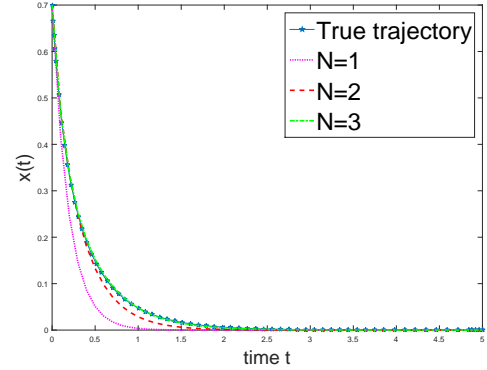} \end{center}
\caption{Successive approximations of the trajectories (as a function of time) of the 1D polynomial system given by $f=-x^3-3x^2-2x$ for $N=1,\hdots,3$. The projection is made on a discretization of the interval $[-0.1,1]$ with a discretization step equal to $0.1.$ Already for $N=3,$ the difference between the trajectory and the approximation is barely noticeable.}
\label{fig-1D}\end{figure}

\subsection{Limit cycles in dimension larger than one}
Our long-term goal is to achieve more complex control tasks with the same kind of approximations, and as a first step beyond stability \rmjj{analysis}, we investigate here the approximability of limit cycles.  This property is quite challenging to reproduce with a linear system, as (nontrivial) limit cycles are numerically unstable for linear systems.  However, we observed that for relatively short time-horizons, the approximations exhibit a cyclic behavior which tends to simulate the limit cycle.  We also observed numerically that, in this case, some eigenvalues of the approximating linear differential equation have a real part significantly close to zero, which is encouraging. We leave a formalization of this observation for further research. 

In Figures \ref{fig:lotka:approx} and \ref{fig:lotka}, we approximated the Lotka-Volterra system with our $N$-polyflow approximation for $N=1,2,5,6.$
\rmjj{More precisely the system is 
\begin{eqnarray}
\label{lotka} \dot x & = & \alpha(x+1)-\beta(x+1)(y+1/2),\\ 
\nonumber \dot y& = &\gamma(x+1)(y+1/2)-\delta(y+1/2),
\end{eqnarray}
       
		with	$ \alpha=2/3;
        \beta=4/3;
        \delta=1;
        \gamma=1.$}	The numerical values are chosen such that a limit cycle exists, which is represented in blue in Figure \ref{fig:lotka}. We observe that the polyflow approximations indeed have an approximate limit cycle (for $n=7,$ probably due to numerical errors, asymptotically the approximation diverges from this cycle).
\begin{figure}
\centering
\begin{tabular}{cc}
\includegraphics[scale = 0.2]{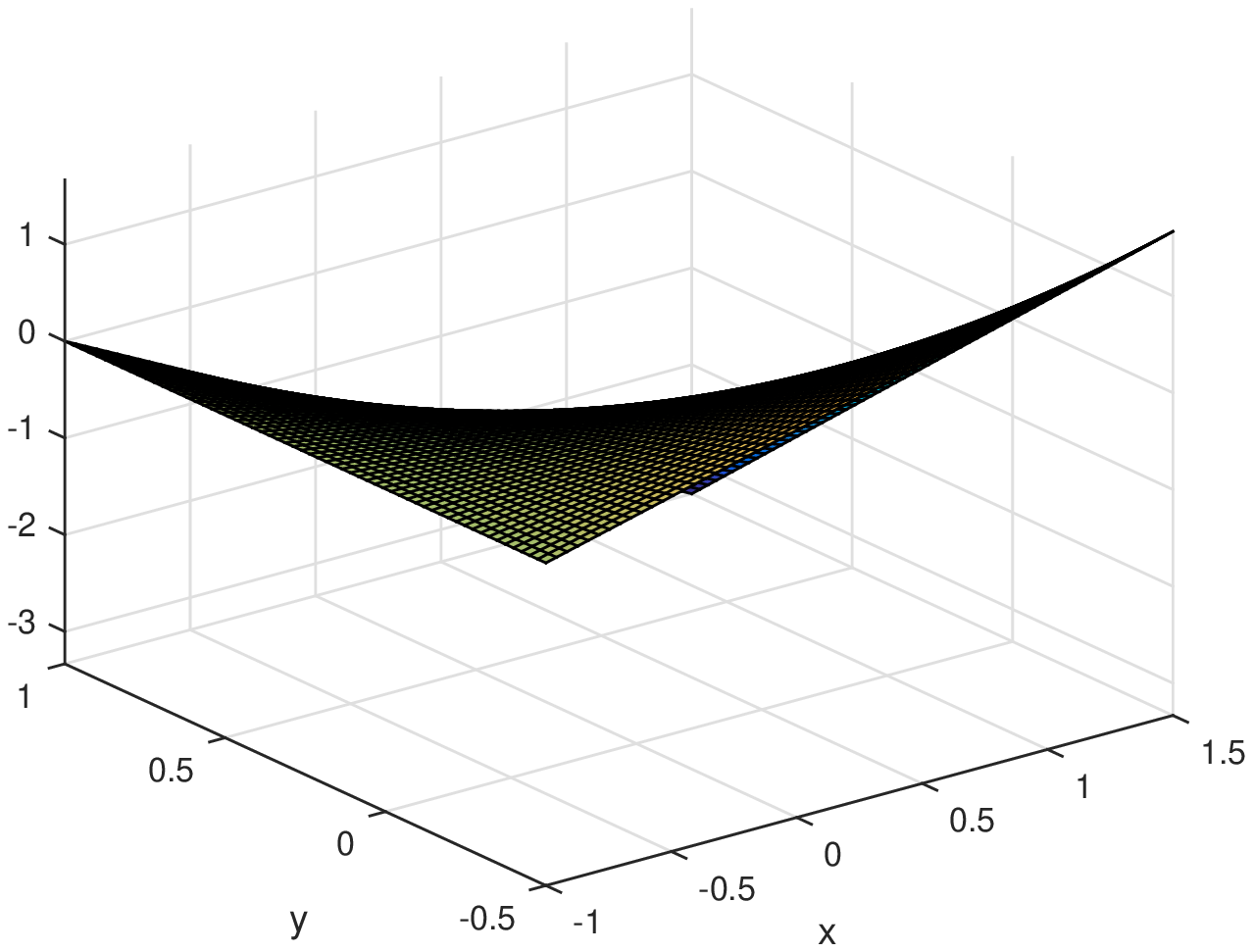}&
\includegraphics[scale = 0.2]{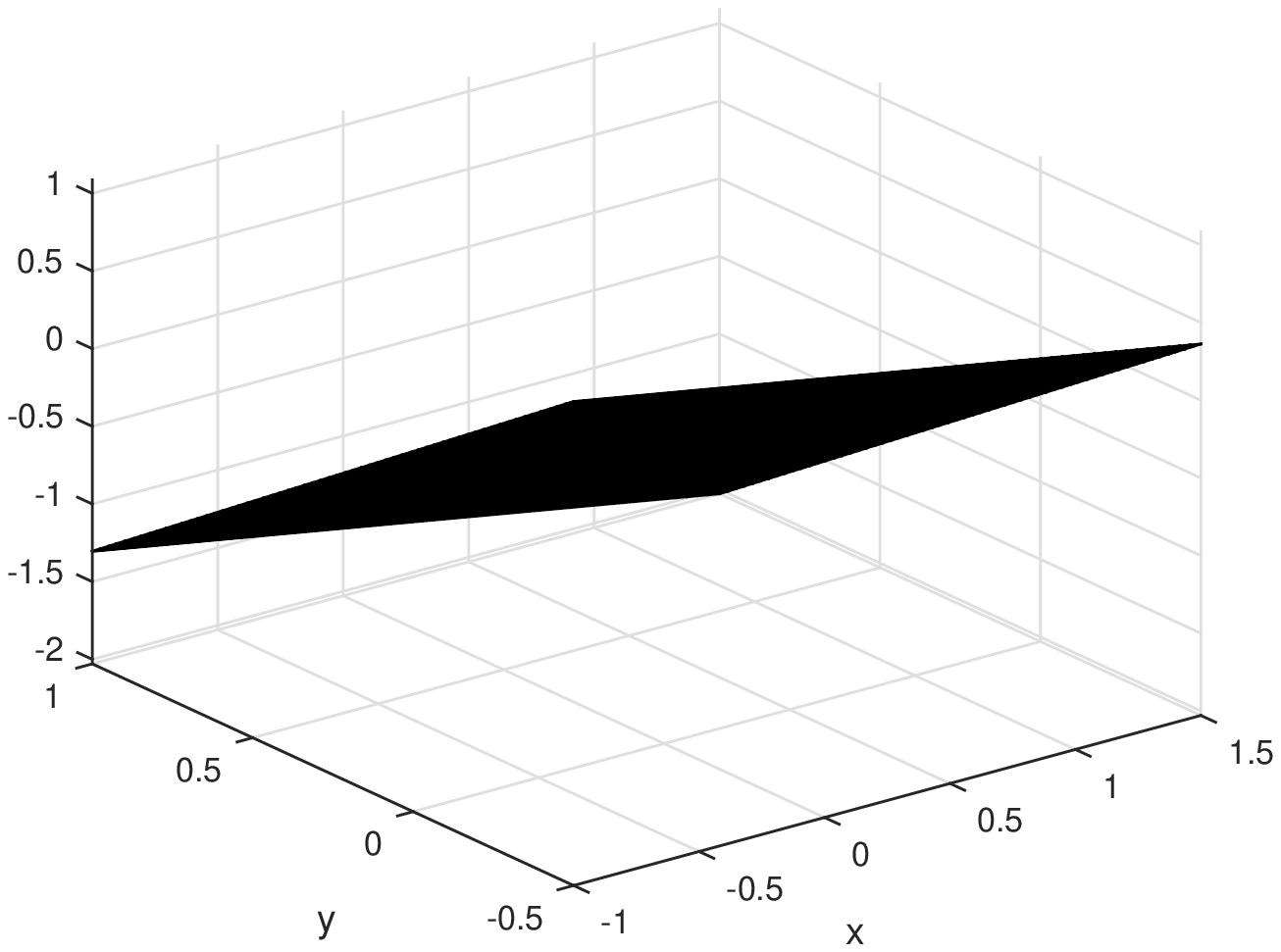} \\$\lieder{1}{f}x_1$ & projection on $\spann_{\re}{\{x_1,x_2\}}$\\ \\
\includegraphics[scale = 0.2]{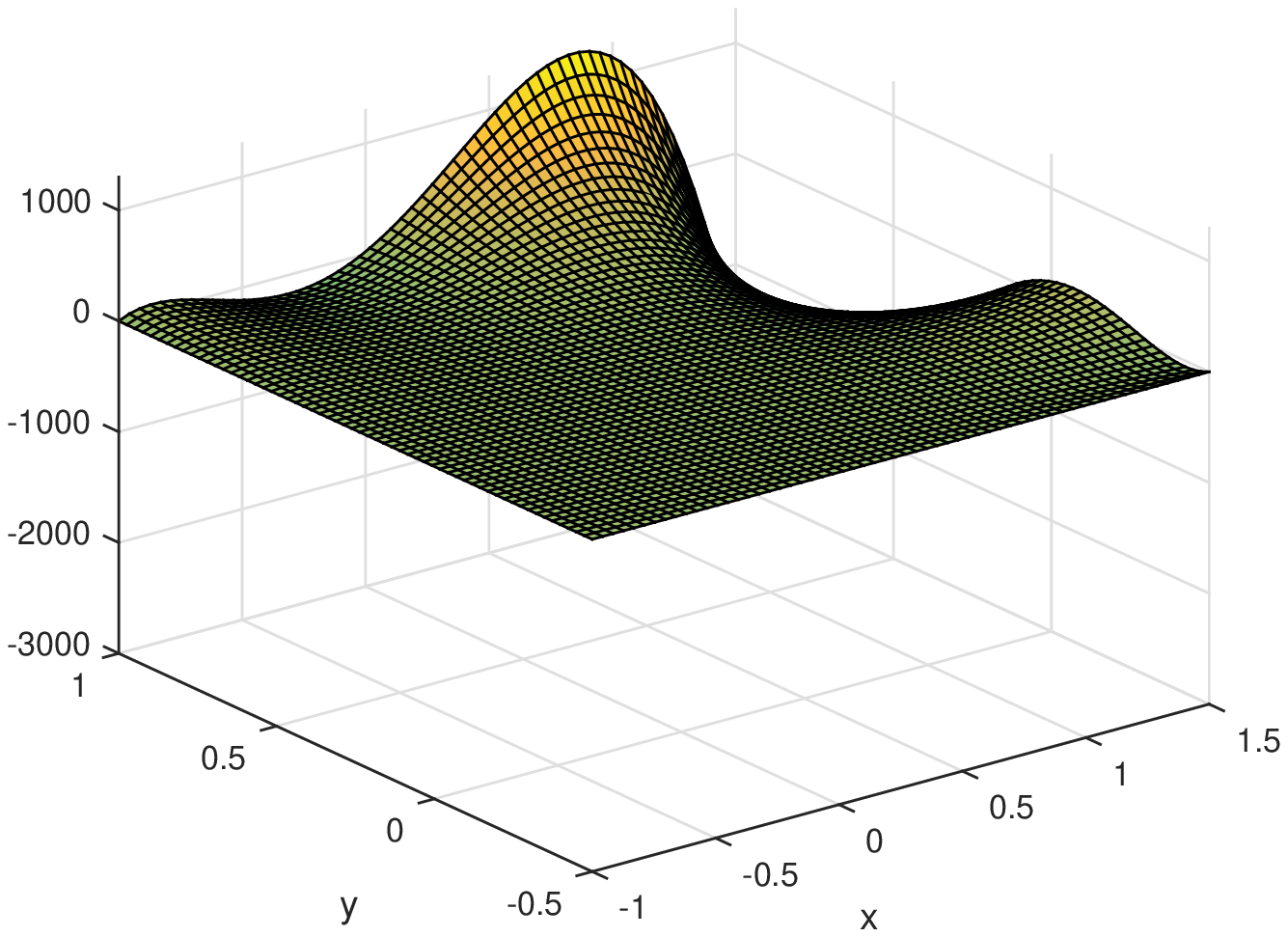}&
\includegraphics[scale = 0.2]{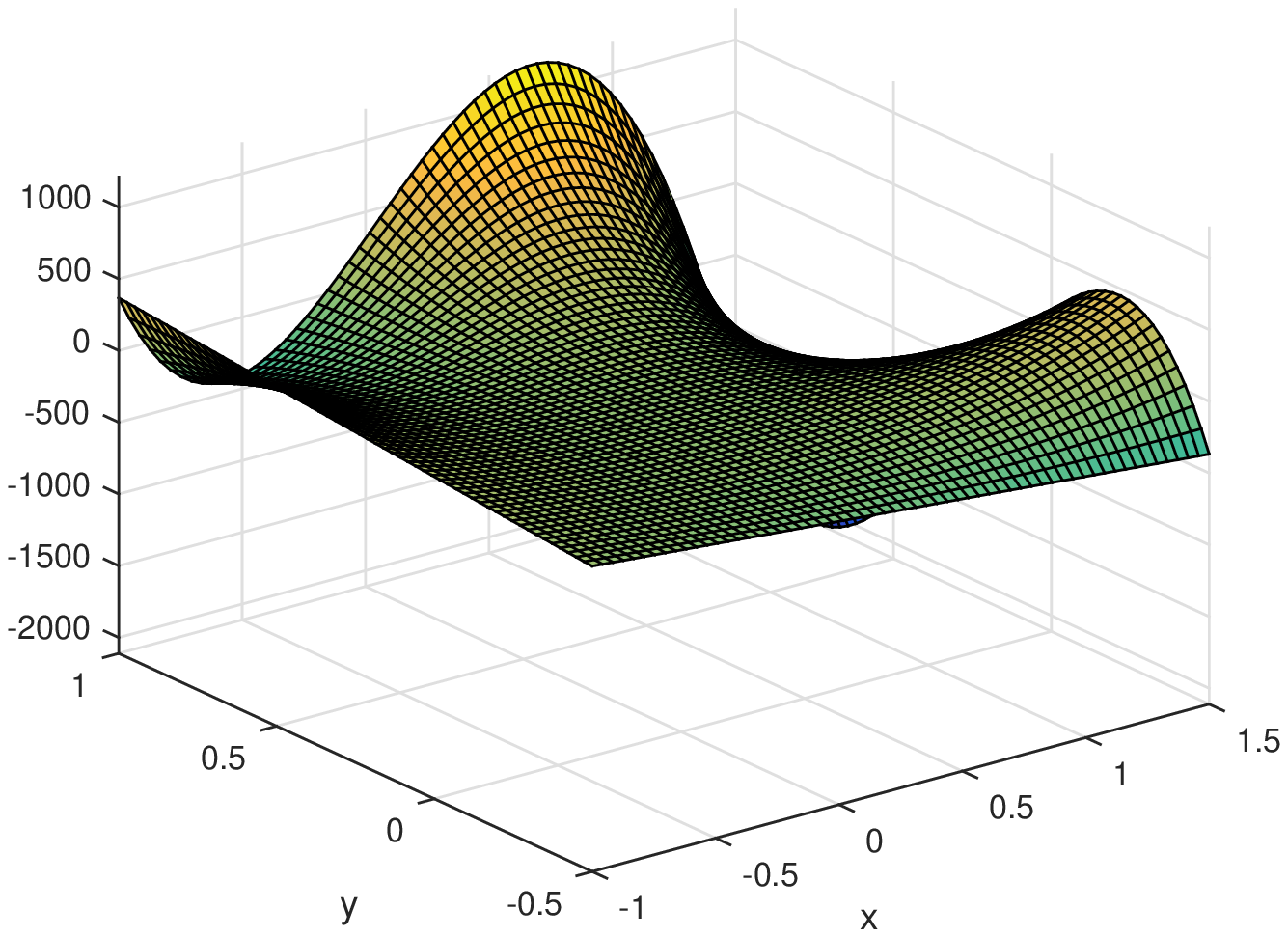}\\$\lieder{6}{f}x_1$ & projection on lower derivatives
\end{tabular}
 \caption{The approximation of the $N$th Lie derivative of one particular coordinate, for $N=1$ and $N=6$.}
\label{fig:lotka:approx}\end{figure}

\begin{figure}
\centering
\begin{tabular}{cc}
\includegraphics[scale = 0.17]{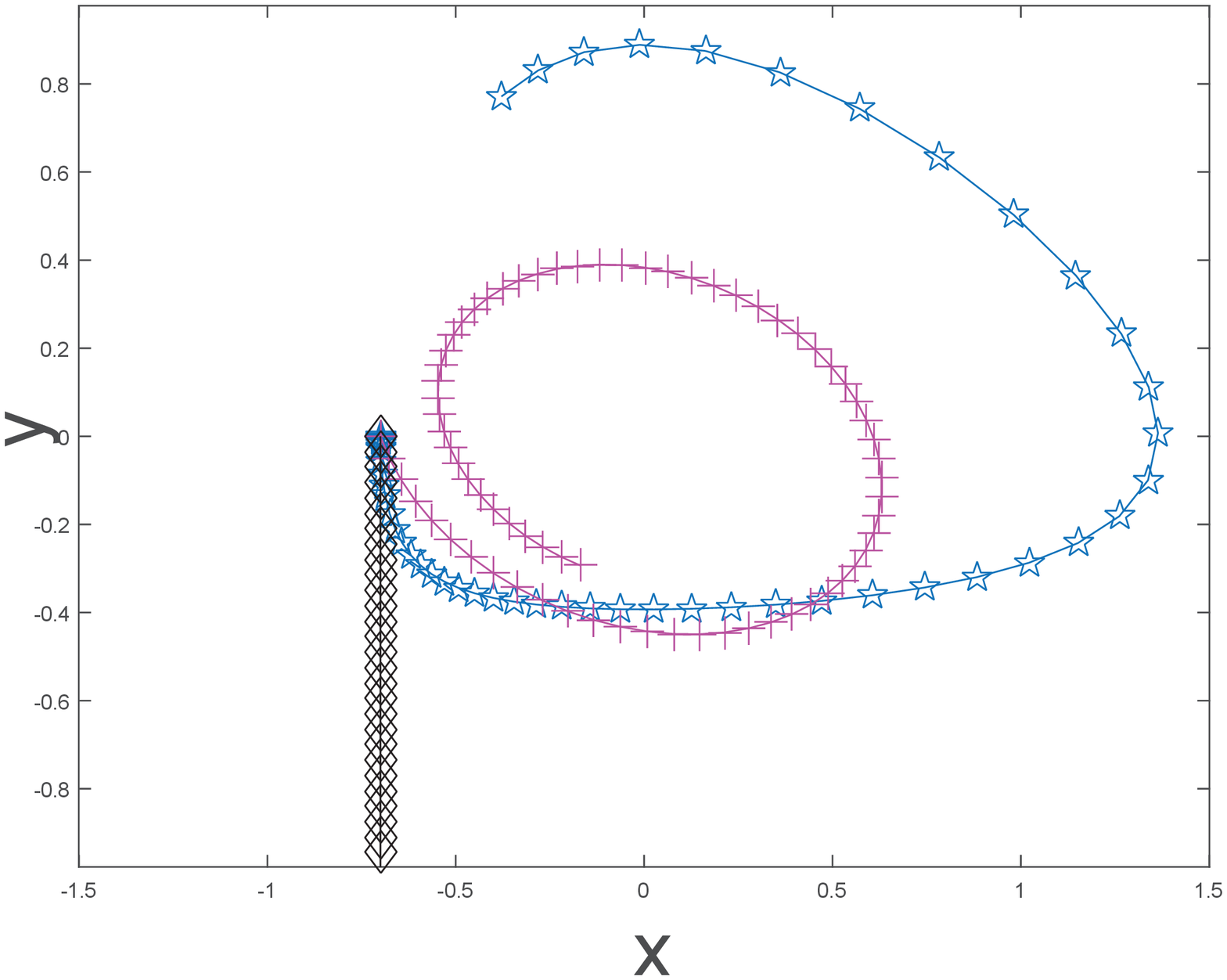}&
\includegraphics[scale = 0.17]{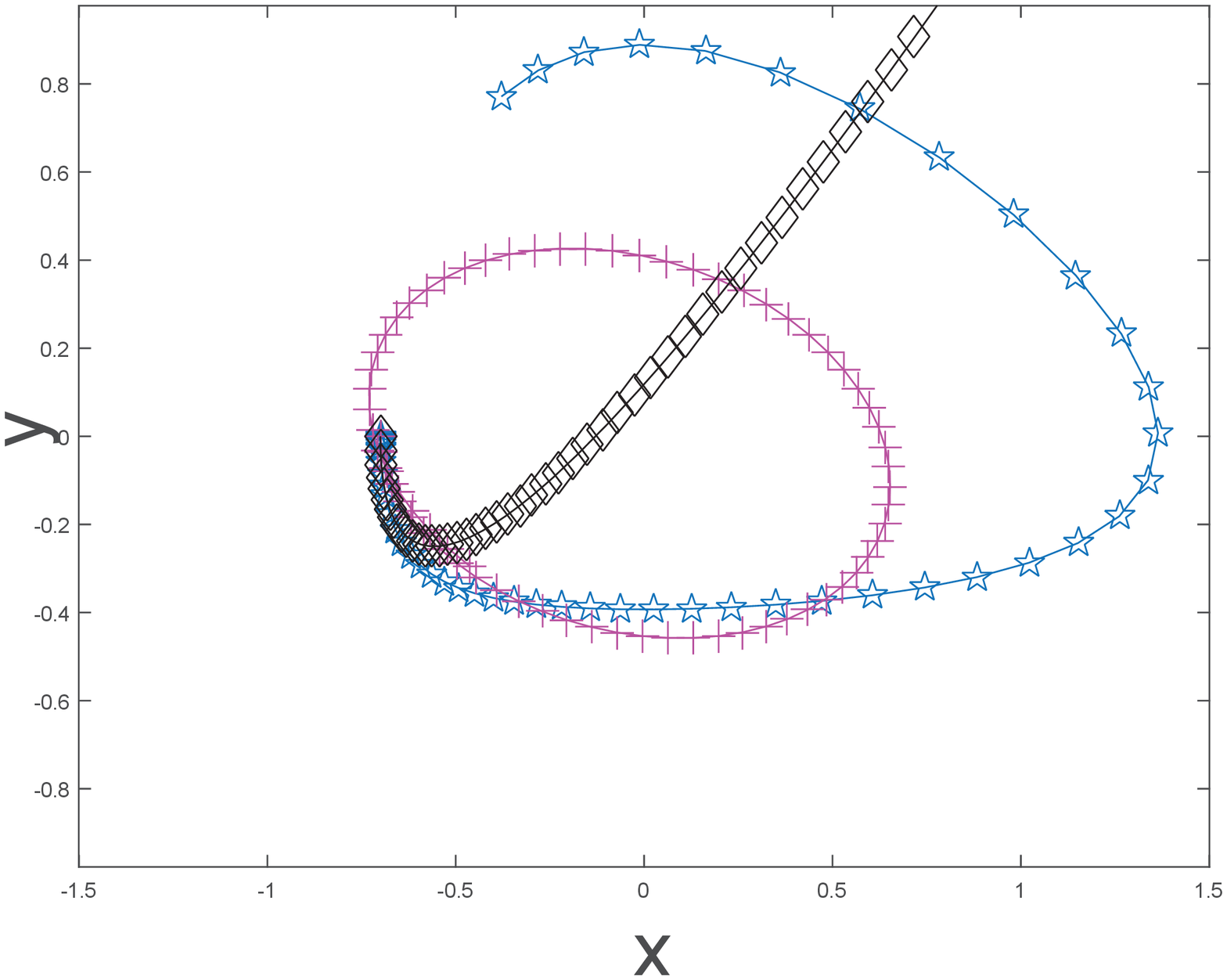} \\(a)&(b)\\ \\
\includegraphics[scale = 0.17]{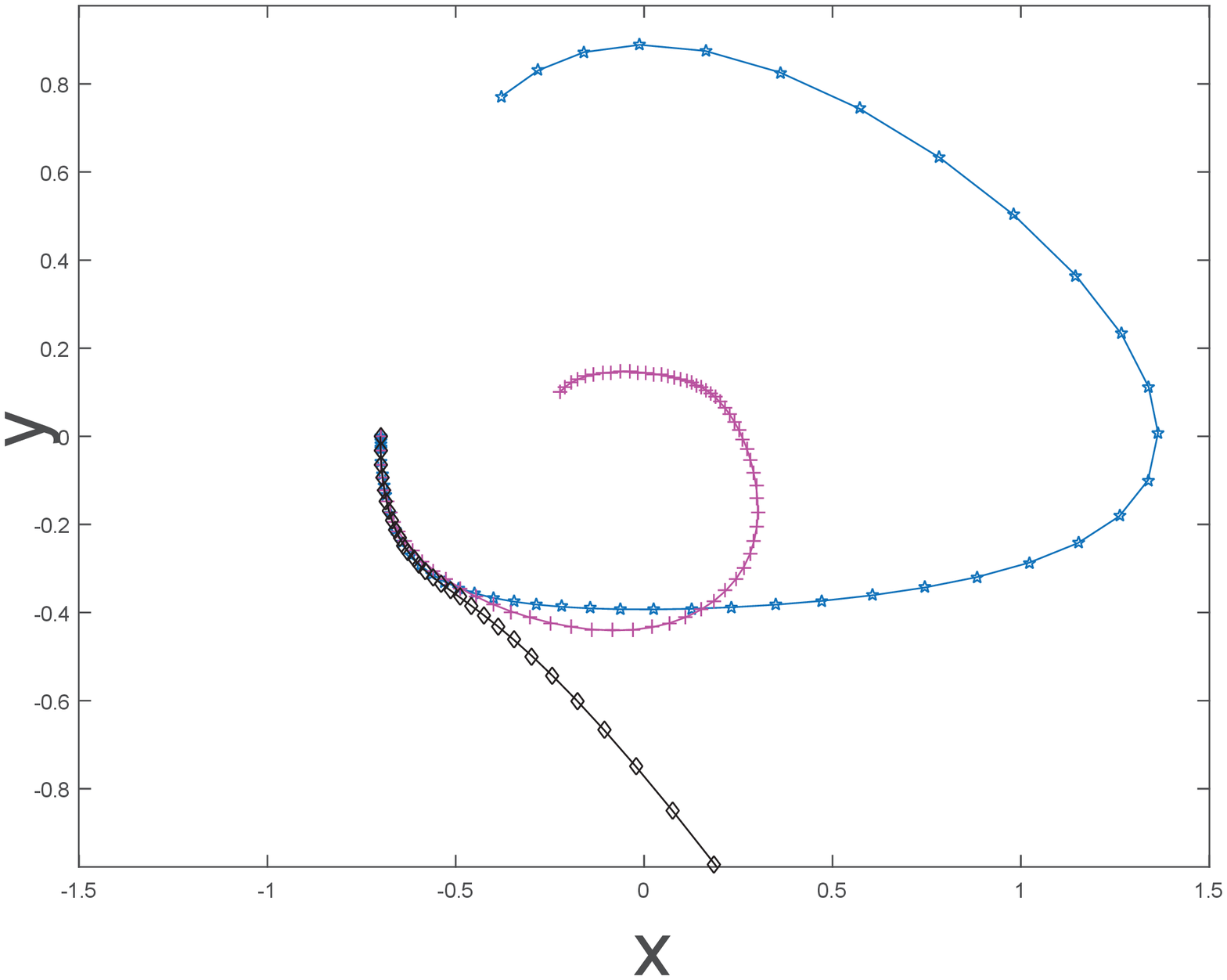}&
\includegraphics[scale = 0.17]{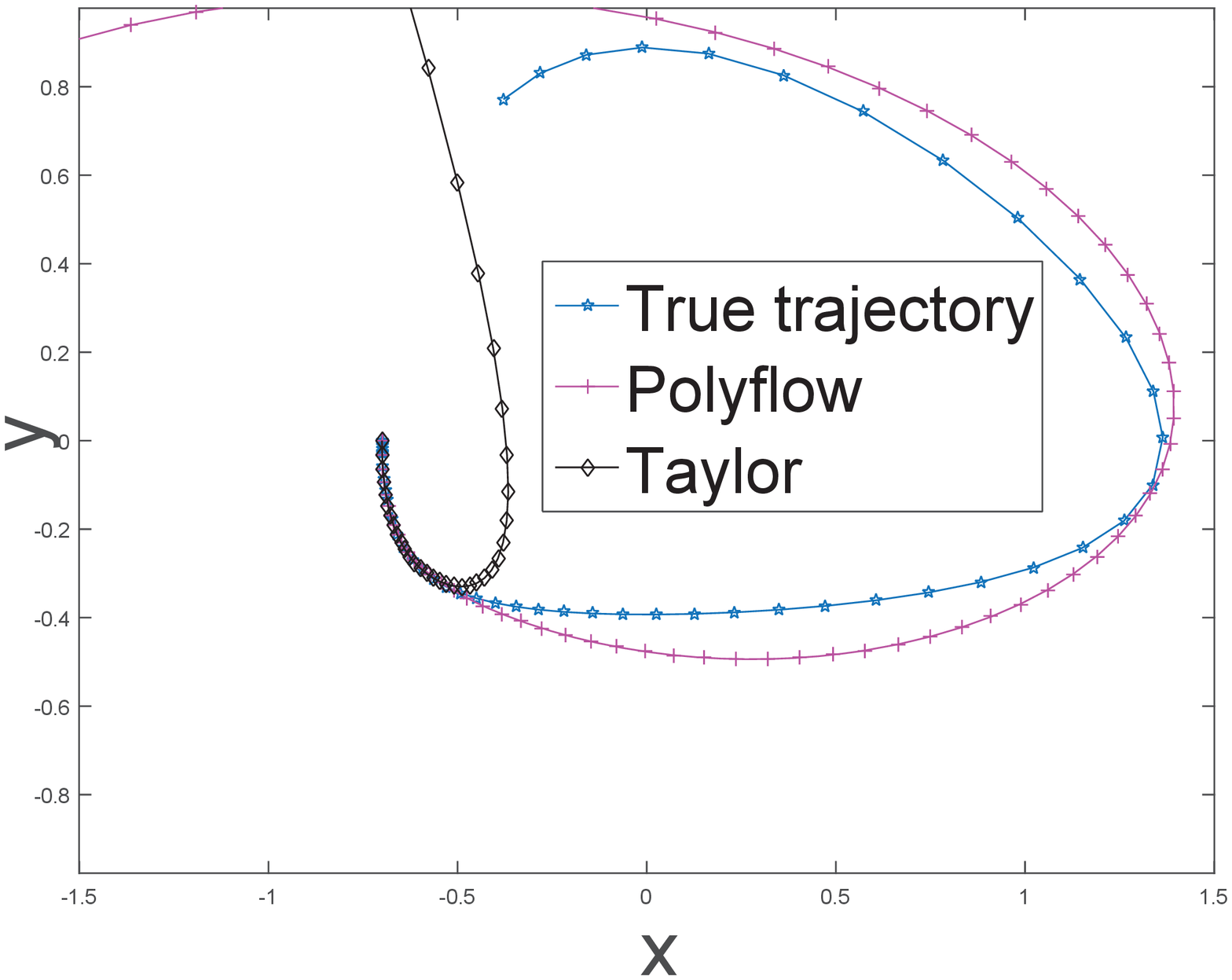}\\(c) & (d)
\end{tabular}
 \caption{The Lotka-Volterra limit cycle and its successive approximations ($N=1,2,6,7$). Comparison of the true trajectory, the Taylor approximation, and the polyflow approximation for several values of $N.$ (The projection for the polyflow approximation is made with respect to the infinity norm on a grid $x\in [-1,1.5];\ y \in[-0.5,1]$ with discretization step equal to $0.1.$)}
\label{fig:lotka}\end{figure}



\section{Conclusion}
In this paper we introduced a natural approach for analysing nonlinear differential equations, by leveraging the notion of `polyflow' from Algebra.
Our technique is nonlocal in two aspects: it is nonlocal in the statespace, in that we approximate the system over a compact set (i.e. not at a particular point), and nonlocal in time, in that one of our goals is to account for asymptotic properties (e.g. asymptotic stability).

Despite its simplicity, we demonstrate that this scheme exhibits very good convergence properties, some of which being theoretically provable.  We hope that such idea of `polyflow approximation' can be pushed further in the future: First, we leave several questions open about the theoretical convergence properties.  Second, we would like to use this technique for more involved objectives than pure stability analysis.  For instance, instead of projecting the nonlinear system on a linear one, one could take the same idea to define an approximate notion of feedback-linearization of nonlinear systems.  Also, the notion of flatness of a nonlinear system, being intrinsically related to the nilpotence of Lie-differential operators, seems well suited for the same type of ideas.
\rmj{Finally, we plan to study the performance of our approach on other systems, and compare it with other techniques, like the Koopman approach.}


\bibliographystyle{ieeetran}        
\bibliography{mybib} 

\begin{thebibliography}{10}
\providecommand{\url}[1]{#1}
\csname url@rmstyle\endcsname
\providecommand{\newblock}{\relax}
\providecommand{\bibinfo}[2]{#2}
\providecommand\BIBentrySTDinterwordspacing{\spaceskip=0pt\relax}
\providecommand\BIBentryALTinterwordstretchfactor{4}
\providecommand\BIBentryALTinterwordspacing{\spaceskip=\fontdimen2\font plus
\BIBentryALTinterwordstretchfactor\fontdimen3\font minus
  \fontdimen4\font\relax}
\providecommand\BIBforeignlanguage[2]{{%
\expandafter\ifx\csname l@#1\endcsname\relax
\typeout{** WARNING: IEEEtran.bst: No hyphenation pattern has been}%
\typeout{** loaded for the language `#1'. Using the pattern for}%
\typeout{** the default language instead.}%
\else
\language=\csname l@#1\endcsname
\fi
#2}}

\bibitem{levine2009analysis}
J.~Levine, \emph{Analysis and control of nonlinear systems: A flatness-based
  approach}.\hskip 1em plus 0.5em minus 0.4em\relax Springer Science \&
  Business Media, 2009.

\bibitem{henrion2005positive}
D.~Henrion and A.~Garulli, \emph{Positive polynomials in control}.\hskip 1em
  plus 0.5em minus 0.4em\relax Springer Science \& Business Media, 2005, vol.
  312.

\bibitem{kowalski1997nonlinear}
K.~Kowalski, ``Nonlinear dynamical systems and classical orthogonal
  polynomials,'' \emph{Journal of Mathematical Physics}, vol.~38, no.~5, pp.
  2483--2505, 1997.

\bibitem{mauroy2016global}
A.~Mauroy and I.~Mezi{\'c}, ``Global stability analysis using the
  eigenfunctions of the koopman operator,'' \emph{IEEE Transactions on
  Automatic Control}, vol.~61, no.~11, pp. 3356--3369, 2016.

\bibitem{kowalski1991nonlinear}
K.~Kowalski and W.-H. Steeb, \emph{Nonlinear dynamical systems and Carleman
  linearization}.\hskip 1em plus 0.5em minus 0.4em\relax World Scientific,
  1991.

\bibitem{khalil1996noninear}
H.~K. Khalil, ``Nonlinear systems,'' \emph{Prentice-Hall, New Jersey}, vol.~2,
  no.~5, pp. 5--1, 1996.

\bibitem{gaude2001solving}
B.~W. Gaude, ``Solving nonlinear aeronautical problems using the carleman
  linearization method,'' \emph{Sand Report}, 2001.

\bibitem{rowley2009spectral}
C.~W. Rowley, I.~Mezi{\'c}, S.~Bagheri, P.~Schlatter, and D.~S. Henningson,
  ``Spectral analysis of nonlinear flows,'' \emph{Journal of fluid mechanics},
  vol. 641, pp. 115--127, 2009.

\bibitem{mauroy2013isostables}
A.~Mauroy, I.~Mezi{\'c}, and J.~Moehlis, ``Isostables, isochrons, and koopman
  spectrum for the action--angle representation of stable fixed point
  dynamics,'' \emph{Physica D: Nonlinear Phenomena}, vol. 261, pp. 19--30,
  2013.

\bibitem{hartman}
P.~Hartman, ``A lemma in the theory of structural stability of differential
  equations,'' \emph{Proceedings of the American Mathematical Society},
  vol.~11, no.~4, pp. 610--620, 1960.

\bibitem{krener1984approximate}
A.~J. Krener, ``Approximate linearization by state feedback and coordinate
  change,'' \emph{Systems \& Control Letters}, vol.~5, no.~3, pp. 181--185,
  1984.

\bibitem{su1982linear}
R.~Su, ``On the linear equivalents of nonlinear systems,'' \emph{Systems \&
  control letters}, vol.~2, no.~1, pp. 48--52, 1982.

\bibitem{FeedbackLin}
B.~Jakubczyk and W.~Respondek, ``On linearization of control systems,''
  \emph{Bull. Acad. Polonaise Sci., Ser. Sci. Math.}, no.~28, pp. 517--522,
  1980.

\bibitem{levine1986nonlinear}
J.~Levine and R.~Marino, ``Nonlinear system immersion, observers and
  finite-dimensional filters,'' \emph{Systems \& Control Letters}, vol.~7,
  no.~2, pp. 133--142, 1986.

\bibitem{van1995locally}
A.~van~den Essen, ``Locally nilpotent derivations and their applications,
  iii,'' \emph{Journal of Pure and Applied Algebra}, vol.~98, no.~1, pp.
  15--23, 1995.

\bibitem{bass1985polynomial}
H.~Bass and G.~Meisters, ``Polynomial flows in the plane,'' \emph{Advances in
  Mathematics}, vol.~55, no.~2, pp. 173--208, 1985.

\bibitem{van1994locally}
A.~Van Den~Essen, ``Locally finite and locally nilpotent derivations with
  applications to polynomial flows, morphisms and 𝒢ₐ-actions. ii,''
  \emph{Proceedings of the American Mathematical Society}, vol. 121, no.~3, pp.
  667--678, 1994.

\bibitem{coomes1991linearization}
B.~Coomes and V.~Zurkowski, ``Linearization of polynomial flows and spectra of
  derivations,'' \emph{Journal of Dynamics and Differential Equations}, vol.~3,
  no.~1, pp. 29--66, 1991.

\bibitem{MTAA}
R.~Abraham, J.~E. Marsden, and T.~Ratiu, \emph{Manifolds, Tensors, Analysis,
  and Application}, 2nd~ed., ser. Applied Mathematical Sciences.\hskip 1em plus
  0.5em minus 0.4em\relax Springer, 1988, vol.~75.

\bibitem{NAT}
D.~Braess, \emph{Nonlinear Approximation Theory}, ser. Computational
  Mathematics.\hskip 1em plus 0.5em minus 0.4em\relax Springer-Verlag, 1986.

\end{thebibliography}

\end{document}